\documentclass[12pt]{article}
\usepackage[english]{babel}
\usepackage[utf8]{inputenc}
\usepackage{amstext}
\usepackage{todonotes}
\usepackage{fancyhdr}
\usepackage{amsfonts,graphicx,bezier, amssymb}
\usepackage{amsmath}
\usepackage{caption}
\usepackage{mathrsfs}
\usepackage{mathtools}
\usepackage{tikz}
\usepackage{rotating}
\usepackage[numbers]{natbib}
\usetikzlibrary{arrows.meta}
\newenvironment{prf}{\noindent{\bf{Proof:}}~~}{\hfill\rule{1ex}{1ex}\vskip1.5ex}

\newcommand{\Z}{\mathbb Z}

\newcommand{\beqa}{\begin{eqnarray}}
\newcommand{\enqa}{\end{eqnarray}}
\newcommand{\beq}{\begin{eqnarray*}}
\newcommand{\enq}{\end{eqnarray*}}

\newcommand{\noi}{\noindent}

\newtheorem{rem}{Remark}[section]

\newtheorem{propn}{Proposition}[section]
\newtheorem{defn}{Definition}[section]
\newtheorem{exam}{{\bf Example}}[section]
\newtheorem{thm}{Theorem}[section]

\newtheorem{lem}{Lemma}[section]

\makeatletter
\providecommand*{\twoheadrightarrowfill@}{%
  \arrowfill@\relbar\relbar\twoheadrightarrow
}
\providecommand*{\twoheadleftarrowfill@}{%
  \arrowfill@\twoheadleftarrow\relbar\relbar
}
\providecommand*{\xtwoheadrightarrow}[2][]{%
  \ext@arrow 0579\twoheadrightarrowfill@{#1}{#2}%
}
\providecommand*{\xtwoheadleftarrow}[2][]{%
  \ext@arrow 5097\twoheadleftarrowfill@{#1}{#2}%
}
\makeatother

\begin{document}
\begin{center}
{\bf\Large   Reduced and coreduced modules with respect to inverse families of ideals}
\end{center}

		 \vspace*{0.2cm}

\begin{center}Sholastica Luambano\footnote{Department of Mathematics, University of Dodoma, P.O. BOX 259, Dodoma, Tanzania} and David Ssevviiri\footnote{Department of Mathematics, Makerere University, P.O. BOX 7062, Kampala, Uganda}$^{,}$\footnote{Corresponding author}  \\
			E-mail:   scholastica.luambano@udom.ac.tz and david.ssevviiri@mak.ac.ug

		\end{center}

	\vspace*{0.1cm}
\begin{abstract}
Let \(R\) be a commutative ring and \(\Phi\) an inverse family of ideals with greatest proper member $\mathcal{L}$. We introduce and study \(\Phi\)-reduced and \(\Phi\)-coreduced modules, extending the corresponding notions for powers of a single ideal. We show that, on these classes of modules, the generalized torsion and completion functors associated with \(\Phi\) are determined by $\mathcal{L}$. We establish characterizations and closure properties of these modules and obtain a Greenlees–May type adjunction and a Matlis–Greenlees–May type characterization. When \(\Phi\) is a system of ideals, we further investigate the radicality of the generalized torsion functor and derive associated torsion theories on suitable Serre subcategories. These results provide a module-theoretic framework for studying generalized torsion and completion with respect to families of ideals.
\end{abstract}

	{\bf Keywords}:  Reduced modules, Coreduced modules, Torsion and Completion, Torsion theory, Greenlees--May Duality and Matlis-Greenlees--May Equivalence

	\vspace*{0.3cm}

{\bf MSC 2020} Mathematics Subject Classification: 13C13, 18A40, 13D30, 13C12, 18E40

\section{Introduction}
\begin{paragraph}\noi
Families of ideals play a fundamental role in commutative algebra and
homological algebra. Perhaps the most familiar example is the descending
family
\[
I \supseteq I^2 \supseteq I^3 \supseteq \cdots
\]
associated with an ideal $I$ of a commutative ring $R$. This family
underlies two classical constructions: the $I$-torsion functor
\[
\Gamma_I(M)
=
\bigcup_{n\geq 1}(0:_M I^n)
\]
and the $I$-adic completion functor
\[
\Lambda_I(M)
=
\varprojlim_n M/I^nM.
\]

Their derived functors give rise to local cohomology and local homology, which occur naturally in duality theory, including Greenlees--May duality
and the Matlis--Greenlees--May equivalence, see for instance \cite{Porta2014}. More generally, many
constructions in local cohomology do not depend essentially on the
powers of a single ideal, but can be developed with respect to suitable inverse families or systems of ideals; see, for example,
\cite{Behrouzian2020, BijanZadeh1980, Brodmann2013, DivaaniAazar2002}.
   
\end{paragraph}

\begin{paragraph}\noi
The purpose of this paper is to show that the theory of reduced and
coreduced modules admits a similar passage from a single ideal to an
inverse family of ideals. This places these module-theoretic notions in
the same general framework as generalized torsion and completion and
reveals connections with torsion theories and classical duality
phenomena.
   
\end{paragraph}

\begin{paragraph}\noi

Recall that an $R$-module $M$ is \emph{reduced} if $ a^2m=0$ implies that $am=0$
for every $a\in R$ and $m\in M$. Reduced modules were introduced by
Lee and Zhou \cite{Lee2004} as a module-theoretic analogue of reduced
rings. In particular, the regular module $R$ is reduced precisely when
$R$ is a reduced ring. The notion has subsequently proved useful in
the study of module-theoretic radicals and related structural questions;
see, for example, \cite{KimuliSsevviiri2023}.
\end{paragraph}

\begin{paragraph}\noi
  
A relative version of this notion was introduced in
\cite{SsevviiriI2024}. Given an ideal $I$ of $R$, an
$R$-module $M$ is called \emph{$I$-reduced} if
\[
I^2m=0
\quad\text{implies that}\quad
Im=0
\]
for every $m\in M$. Dually, $M$ is called \emph{$I$-coreduced} if
\[
I^2M=IM.
\]
These notions are closely connected with the torsion and completion
functors associated with $I$. For instance, $I$-reducedness is
equivalent to the stabilization
\[
(0:_M I)=(0:_M I^n)
\qquad (n\geq 1),
\]
and hence to
\[
\Gamma_I(M)=(0:_M I).
\]
Similarly, $I$-coreducedness expresses stabilization on the completion
side:
\[
IM=I^nM
\qquad (n\geq 1).
\]
  
\end{paragraph}

\begin{paragraph}\noi  
Thus, reducedness and coreducedness may be viewed as conditions under
which the towers defining torsion and completion stabilize at their
first stage. This perspective has led to applications involving
torsion theories, radical functors and Greenlees--May type phenomena;
see \cite{SsevviiriI2024,  SsevviiriII2026, SsevviiriIII2025}. 
 
\end{paragraph}

\begin{paragraph}\noi

The powers $\{I^n\}_{n\geq 1}$, however, constitute only one particular
inverse family of ideals. In local cohomology, it has long been useful
to replace this family by a more general directed family
\[
\Phi=\{I_\alpha\}_{\alpha\in\Lambda},
\qquad
\alpha\geq\beta
\quad\Longrightarrow\quad
I_\alpha\subseteq I_\beta,
\]
and to consider the corresponding generalized torsion functor
\[
\Gamma_\Phi(M)
=
\{m\in M : I_\alpha m=0
\text{ for some }\alpha\}
=
\bigcup_{I_\alpha\in\Phi}(0:_M I_\alpha).
\]
Its natural completion counterpart is
\[
\Lambda_\Phi(M)
=
\varprojlim_{I_\alpha\in\Phi}M/I_\alpha M.
\]
   
\end{paragraph}

\begin{paragraph}\noi    
 
This raises the following natural question:

\begin{quote}
\emph{Can reduced and coreduced modules be formulated directly with
respect to an inverse family of ideals, and if so, how do the resulting
notions interact with generalized torsion, completion and torsion
theory?}
\end{quote}

We answer this question by introducing \emph{$\Phi$-reduced} and
\emph{$\Phi$-coreduced modules}. Suppose that $\Phi$ has a greatest
proper member $\mathcal{L}$ with respect to inclusion. Roughly speaking, an
$R$-module $M$ is $\Phi$-reduced when annihilation by any member of
$\Phi$ already forces annihilation by $\mathcal{L}$, whereas it is
$\Phi$-coreduced when multiplication by every member of $\Phi$ has
the same image as multiplication by $\mathcal{L}$. When
$\Phi=\{I^n\}_{n\geq 1}$, these notions recover precisely the previously studied $I$-reduced
and $I$-coreduced modules. 
\end{paragraph}

\begin{paragraph}\noi
The main point is that these definitions are not merely formal generalizations. They identify classes of modules on which the
generally more complicated functors $\Gamma_\Phi$ and $\Lambda_\Phi$
acquire particularly simple descriptions. For a $\Phi$-reduced module
$M$, we obtain
\[
\Gamma_\Phi(M)
=
(0:_M \mathcal{L})
\cong
\operatorname{Hom}_R(R/\mathcal{L},M),
\]
while for a $\Phi$-coreduced module $M$,
\[
\Lambda_\Phi(M)
\cong
M/\mathcal{L}M
\cong
R/\mathcal{L}\otimes_R M.
\]
   
\end{paragraph}

\begin{paragraph}\noi
Thus, on these classes of modules, generalized torsion and completion
associated with an entire inverse family are controlled by a
\emph{single ideal} $\mathcal{L}$. This reduction from a family of ideals to one
distinguished ideal is one of the central structural features of the
theory developed here. The corresponding characterizations are
established in Propositions \ref{p2} and~\ref{dualidealsyst1}.

\end{paragraph}

\begin{paragraph}\noi

This simplification also exposes a natural duality. On the subcategories
of $\Phi$-reduced and $\Phi$-coreduced modules, respectively, the
functors $\Gamma_\Phi$ and $\Lambda_\Phi$ reduce to the familiar
adjoint pair
\[
R/\mathcal{L}\otimes_R -
\quad\dashv\quad
\operatorname{Hom}_R(R/\mathcal{L},-).
\]
We consequently obtain a Greenlees--May type adjunction
\[
\operatorname{Hom}_R(\Lambda_\Phi(M),N)
\cong
\operatorname{Hom}_R(M,\Gamma_\Phi(N))
\]
for $\Phi$-coreduced $M$ and $\Phi$-reduced $N$. We also identify a
corresponding Matlis--Greenlees--May type phenomenon: for an
$R$-module $M$, being simultaneously $\Phi$-torsion and
$\Phi$-reduced, being simultaneously $\Phi$-complete and
$\Phi$-coreduced and satisfying
$\mathcal{L}M=0$
are equivalent conditions.
   
\end{paragraph}

\begin{paragraph}\noi
A further advantage of working with families of ideals becomes visible
when $\Phi$ satisfies the stronger multiplicative condition of a
\emph{system of ideals}. Such systems occur naturally in generalized
local cohomology and allow the functor $\Gamma_\Phi$ to interact well
with products of ideals. Under this hypothesis, we investigate
radicality and torsion-theoretic properties of $\Gamma_\Phi$. In
particular, on a Serre subcategory of $\Phi$-reduced modules,
$\Gamma_\Phi$ becomes a left exact radical. We obtain torsion theories on suitable
Serre subcategories of $\Phi$-reduced and $\Phi$-coreduced modules.
These results extend the corresponding theory for a fixed ideal and
connect with the classical torsion--torsionfree framework of Jans \cite{Jans1965}.
   
\end{paragraph}

\begin{paragraph}\noi
The framework is also useful from another perspective. Many familiar
results concerning torsion and completion are most commonly formulated
under Noetherian or finite-generation hypotheses. The reduced/coreduced
conditions developed here provide an alternative mechanism under which
analogous simplifications can occur even for modules that are not
finitely generated and over rings that need not be Noetherian. In this
sense, $\Phi$-reducedness and $\Phi$-coreducedness isolate
module-theoretic conditions that compensate, for the purposes considered
here, for some of the finiteness assumptions appearing in the classical
theory.
\end{paragraph}

\begin{paragraph}\noi
The paper is organized as follows. In Section~2 we introduce
$\Phi$-reduced modules, characterize them through the generalized
torsion functor $\Gamma_\Phi$ and establish their basic closure and
comparison properties. Section~3 develops the dual theory of
$\Phi$-coreduced modules and characterizes them in terms of the
generalized completion functor $\Lambda_\Phi$. We also study the
relationship between the reduced and coreduced conditions through
Hom--tensor duality. Section~4 establishes the Greenlees--May type
adjunction and the corresponding Matlis--Greenlees--May type
equivalence. Finally, in Section~5 we specialize to systems of ideals
and investigate radicality and torsion theories, culminating in
torsion-theoretic results on Serre subcategories of $\Phi$-reduced and
$\Phi$-coreduced modules.   
\end{paragraph}

\section{$\Phi$-reduced modules}

 \begin{paragraph}\noi We first recall the notion of an inverse family of ideals.     
 \end{paragraph}

\begin{defn}\rm {\cite[Definition 1.2.10]{Brodmann2013}}   Let  $(\Lambda, \leq)$ be a partially ordered (non-empty) directed set. By an {\it inverse family of ideals} (of $R$) over $\Lambda$, we mean a family $(I_{\alpha})_{\alpha\in \Lambda}$ of ideals of $R$ such that whenever $(\alpha, \beta) \in \Lambda \times \Lambda$ with $\alpha \geq \beta$,
we have $I_{\alpha} \subseteq I_{\beta}$. 
\end{defn}

 \begin{paragraph}\noi
Throughout this paper, unless otherwise stated, $\Phi$ denotes an
inverse family of ideals of $R$ having a greatest proper member $\mathcal{L}$
with respect to inclusion. Thus, $\mathcal{A}\subseteq \mathcal{L}$ for every $\mathcal{A}\in\Phi$. 
\end{paragraph}

\begin{exam}\rm 
 Let $I$ and $J$ be ideals of a ring $R$. The sets $\{I, I^2, I^3, \cdots\}$, \\ $\{I, IJ, I^2J^2, I^3J^3, \cdots\}$,  $\{I, I^2J, I^3J^2, I^4J^3, I^5J^4, \cdots\}$ and $\{I^2J, I^4J^2, I^6J^3, \cdots \}$ are examples of inverse families of ideals of $R$.    
\end{exam} 

\begin{exam}\rm  
Let \(R := k[x,y]\) be the polynomial ring over a field \(k\), and let \(\Lambda = \mathbb{N}^2\), where $\mathbb{N}=\{1, 2, 3, \cdots\}$ with the product order: \((a,b) \le (c,d)\) if and only if \(a\le c\) and \(b\le d\). For each \((a,b)\in\Lambda\),  define $I_{(a,b)} := (x^a y^b)$.
If \((a,b), (c,d)\in\Lambda\), choose \(\delta = (a+c,\; b+d)\). Then \(\delta\ge (a,b)\),  \(\delta\ge (c,d)\), and $I_\delta = (x^{a+c}y^{b+d}) = (x^a y^b)(x^c y^d) = I_{(a,b)} I_{(c,d)}$. If \((a,b)\ge (c,d)\),  then $(x^a y^b) \subseteq (x^c y^d)$, so $I_{(a,b)}\subseteq I_{(c,d)}$. 
\end{exam}

\begin{defn}\rm
 Let $\Phi$ be an inverse family of ideals of a ring $R$ having a greatest proper member $\mathcal{L}$ with respect to inclusion. An $R$-module $M$ is {\it $\Phi$-reduced} if for every $m \in M$ and ideal $\mathcal{A} \in \Phi$, $\mathcal{A}m = 0$ implies that $\mathcal{L}m = 0$.

\end{defn}

\begin{propn}\rm\label{p2}
 Let $M$ be an $R$-module, and let $\Phi$ be an inverse family of ideals of $R$ having a greatest proper member $\mathcal L$ with respect to inclusion. Then the following statements are equivalent:
\begin{enumerate}
    \item $M$ is $\Phi$-reduced;

    \item $    (0:_M\mathcal A)=(0:_M\mathcal L)$   for every $\mathcal A\in\Phi$;

    \item for every $\mathcal A\in\Phi$, the canonical inclusion     $(0:_M\mathcal L)\longrightarrow (0:_M\mathcal A)$
    is an isomorphism;

    \item for every $\mathcal A\in\Phi$, the canonical morphism
    $   \operatorname{Hom}_R(R/\mathcal L,M)  \longrightarrow     \operatorname{Hom}_R(R/\mathcal A,M)$,     
    induced by the canonical epimorphism  $ R/\mathcal A\longrightarrow R/\mathcal L$,   is an isomorphism;

    \item     $\Gamma_\Phi(M)=(0:_M\mathcal L)$; 
    \item $\mathcal L\Gamma_\Phi(M)=0$. 
    \end{enumerate}
Consequently, if these equivalent conditions hold, then there is a
natural isomorphism
$\Gamma_\Phi(M)\cong \operatorname{Hom}_R(R/\mathcal L,M)$.
\end{propn}

\begin{prf}
Since $\mathcal L$ is a greatest proper member of $\Phi$, we have
$\mathcal A\subseteq\mathcal L$ for every $\mathcal A\in\Phi$. Consequently,
$(0:_M\mathcal L)\subseteq(0:_M\mathcal A)$ for every $\mathcal A\in\Phi$. We prove the equivalences.

\medskip
\noindent
\textbf{$(1)\Rightarrow(2)$.}
Suppose that $M$ is $\Phi$-reduced. Let $\mathcal A\in\Phi$.
As observed above, $(0:_M\mathcal L)\subseteq(0:_M\mathcal A)$. Conversely, if $m\in(0:_M\mathcal A)$, then $\mathcal Am=0$.  Since $M$ is $\Phi$-reduced, it follows that
$\mathcal Lm=0$. Thus $m\in(0:_M\mathcal L)$, and hence $(0:_M\mathcal A)=(0:_M\mathcal L)$. 
\medskip
\noindent
\textbf{$(2)\Rightarrow(3)$.} For every $\mathcal A\in\Phi$, the inclusion
$(0:_M\mathcal L)\hookrightarrow(0:_M\mathcal A)$ is canonical. By (2), its source and target are equal. Hence this canonical inclusion is an isomorphism.

\medskip
\noindent
\textbf{$(3)\Rightarrow(4)$.}
For every ideal $\mathcal I$ of $R$, there is a natural isomorphism
\[
\operatorname{Hom}_R(R/\mathcal I,M)
\cong
(0:_M\mathcal I),
\]
given by
\[
f\longmapsto f(1+\mathcal I).
\]
Since $\mathcal A\subseteq\mathcal L$, the canonical epimorphism
$q_{\mathcal A}:R/\mathcal A\longrightarrow R/\mathcal L$ 
induces, by precomposition, a canonical morphism
\[
\operatorname{Hom}_R(R/\mathcal L,M)
\longrightarrow
\operatorname{Hom}_R(R/\mathcal A,M).
\]
Under the natural identifications $\operatorname{Hom}_R(R/\mathcal L,M) \cong(0:_M\mathcal L)$  and $\operatorname{Hom}_R(R/\mathcal A,M) \cong(0:_M\mathcal A)$, this morphism corresponds precisely to the canonical inclusion
\[
(0:_M\mathcal L)\hookrightarrow(0:_M\mathcal A).
\]
Therefore, by (3), the induced Hom morphism is an isomorphism.

\medskip
\noindent
\textbf{$(4)\Rightarrow(2)$.}
By the same natural identifications, the canonical morphism
\[
\operatorname{Hom}_R(R/\mathcal L,M)
\longrightarrow
\operatorname{Hom}_R(R/\mathcal A,M)
\]
corresponds to the canonical inclusion
\[
(0:_M\mathcal L)\hookrightarrow(0:_M\mathcal A).
\]
If the former is an isomorphism, then so is the latter. Hence
\[
(0:_M\mathcal A)=(0:_M\mathcal L)
\]
for every $\mathcal A\in\Phi$. Thus (2) holds.

\medskip
\noindent
\textbf{$(2)\Rightarrow(5)$.}
By definition, $\Gamma_\Phi(M) = \bigcup_{\mathcal A\in\Phi}(0:_M\mathcal A)$. By (2),
$(0:_M\mathcal A)=(0:_M\mathcal L)$ for every $\mathcal A\in\Phi$. Therefore
$\Gamma_\Phi(M) = (0:_M\mathcal L)$. 

\medskip
\noindent
\textbf{$(5)\Rightarrow(6)$.}
If $\Gamma_\Phi(M)=(0:_M\mathcal L)$, then $\mathcal L\Gamma_\Phi(M)
= \mathcal L(0:_M\mathcal L) = 0$. 

\medskip
\noindent
\textbf{$(6)\Rightarrow(1)$.}
Suppose that
$\mathcal L\Gamma_\Phi(M)=0$. Let $\mathcal A\in\Phi$ and let $m\in M$ satisfy
$\mathcal Am=0$. Then $m\in(0:_M\mathcal A)\subseteq\Gamma_\Phi(M)$. Hence, by (6), $\mathcal Lm=0$. Therefore $\mathcal Am=0$ implies that $\mathcal Lm=0$ for every $\mathcal A\in\Phi$ and every $m\in M$. Thus $M$ is $\Phi$-reduced. Finally, under these  six equivalent conditions,
$\Gamma_\Phi(M)=(0:_M\mathcal L)$ and the standard natural isomorphism
$(0:_M\mathcal L) \cong \operatorname{Hom}_R(R/\mathcal L,M)$ gives
$\Gamma_\Phi(M)\cong \operatorname{Hom}_R(R/\mathcal L,M)$. 
\end{prf}

\begin{propn}\rm
The class of $\Phi$-reduced $R$-modules is closed under submodules,
arbitrary direct products, arbitrary direct sums and inverse limits.
\end{propn}

\begin{prf}
Closure under submodules follows immediately from the definition. Let $\{M_{\alpha}\}_{\alpha\in \Lambda}$ be a family of $\Phi$-reduced modules and let
$M=\prod_{\alpha\in \Lambda}M_{\alpha}$. Suppose that $\mathcal{A}\in\Phi$ and $m=(m_{\alpha})_{\alpha\in \Lambda}\in M$ satisfy $\mathcal{A}m=0$. Then $\mathcal{A}m_{\alpha}=0$ for every $\alpha$. Since each $M_{\alpha}$ is $\Phi$-reduced,
$\mathcal{L}m_{\alpha}=0$ for every $\alpha$, and hence $\mathcal{L}m=0$. Thus $M$ is
$\Phi$-reduced. Since $\bigoplus_{\alpha\in\Lambda}M_{\alpha}$ is a submodule of
$\prod_{\alpha\in \Lambda}M_{\alpha}$, arbitrary direct sums are also
$\Phi$-reduced. Finally, an inverse limit of $R$-modules is naturally a submodule of
the corresponding direct product. Therefore inverse limits of
$\Phi$-reduced modules are $\Phi$-reduced.
\end{prf}

\begin{paragraph}\noindent
If $I$ is an ideal of $R$ and $\Phi=\{I^n\}_{n\in \Z^+}$, then $\Phi$-reduced $R$-modules are just the $I$-reduced $R$-modules and Proposition \ref{p2} becomes just Proposition \ref{p3}.
\end{paragraph}

\begin{propn}\rm\cite[Proposition 2.2]{SsevviiriI2024}\label{p3}
  For any $R$-module $M$ and an ideal $I$ of $R$, the following statements are equivalent:
   \begin{enumerate}
       \item $M$ is $I$-reduced,
       \item $(0:_M I)=(0:_M I^2)$,
\item $\text{Hom}_{R}(R/I,M)\cong \text{Hom}_R(R/I^2, M )$,
       \item $\varGamma_I(M)\cong \text{Hom}_R(R/I,M)$,
       \item $I\varGamma_I(M)=0$.
   \end{enumerate}
\end{propn}
 
\begin{propn}\rm\label{2-idealsyst}
Let $\Phi_1$ and $\Phi_2$ be inverse families of ideals of $R$
having the same greatest proper member $\mathcal L$.

\begin{enumerate}
    \item Suppose that for every $\mathcal{B}\in\Phi_2$ there exists
    $\mathcal{A}\in\Phi_1$ such that $\mathcal{A}\subseteq \mathcal{B}$. If $M$ is     $\Phi_1$-reduced, then $M$ is $\Phi_2$-reduced. Moreover,
   $\Gamma_{\Phi_1}(M)=\Gamma_{\Phi_2}(M)=(0:_M \mathcal{L})$. 

    \item Suppose that for every $\mathcal{A}\in\Phi_1$ there exists
    $\mathcal{B}\in\Phi_2$ such that $\mathcal{B}\subseteq \mathcal{A}$. If $M$ is
    $\Phi_2$-reduced, then $M$ is $\Phi_1$-reduced. Moreover,
   $\Gamma_{\Phi_1}(M)=\Gamma_{\Phi_2}(M)=(0:_M \mathcal{L})$.
\end{enumerate}
\end{propn}

\begin{prf}
We prove (1); the proof of (2) is symmetric. Let $m\in \Gamma_{\Phi_2}(M)$. Then there exists $\mathcal{B}\in\Phi_2$ such that $\mathcal{B}m=0$. By hypothesis, there exists $\mathcal{A}\in\Phi_1$
such that $\mathcal{A}\subseteq \mathcal{B}$. Hence $\mathcal{A}m=0$,
and therefore $    m\in\Gamma_{\Phi_1}(M)$. Thus
$  \Gamma_{\Phi_2}(M)\subseteq\Gamma_{\Phi_1}(M)$. 
Since $M$ is $\Phi_1$-reduced, Proposition~\ref{p2} gives
$\Gamma_{\Phi_1}(M)=(0:_M \mathcal{L})$. On the other hand, since $\mathcal{L}\in\Phi_2$, $ (0:_M \mathcal{L})\subseteq\Gamma_{\Phi_2}(M)$. 
Consequently, $(0:_M \mathcal{L})
    \subseteq\Gamma_{\Phi_2}(M)
    \subseteq\Gamma_{\Phi_1}(M)
    =(0:_M \mathcal{L})$. Hence $\Gamma_{\Phi_2}(M)=(0:_M \mathcal{L})$.
By Proposition~\ref{p2}, $M$ is $\Phi_2$-reduced. Moreover,
$\Gamma_{\Phi_1}(M)=\Gamma_{\Phi_2}(M)=(0:_M \mathcal{L})$. 
\end{prf}

\begin{propn}\rm \label{phi vs I} Let $I$ be an ideal of a ring $R$. Suppose that $\Phi=\{ J_1, J_2, J_3,\cdots\}$ is an inverse family of ideals of a ring $R$  such that $ J_1 \supseteq J_2 \supseteq J_3\supseteq \cdots$ and $I=J_1$. If  $ J_n\subseteq I^n ~(\text{resp}.~ I^n \subseteq J_n)$ for all positive integers $n \geq 2$, then whenever $M$ is  $\Phi$-reduced (resp. $I$-reduced), it follows that $M$ is also $I$-reduced (resp. $\Phi$-reduced). In both cases, $\Gamma_\Phi(M)=\Gamma_I(M)$.
    \end{propn}
\begin{prf}\rm
    Proposition \ref{phi vs I} is a special case of Proposition \ref{2-idealsyst}.
\end{prf}

\begin{exam}\rm
Let $R$ be $\mathbb{Z}$, the ring of integers and  $I$ be the ideal $2\mathbb{Z}$ of $\mathbb{Z}$. If  $\Phi=\{2^n\mathbb{Z}\}_{n\geq 1}$ and $M$ is the $\mathbb{Z}$-module $\mathbb{Z}/6\mathbb{Z}$ then, $M$ is both $I$-reduced and $\Phi$-reduced. Note that $\Gamma_\Phi(M)=(0:_M 2\mathbb{Z})=\{0,3\}=\Gamma_I(M)$.
\end{exam}

\begin{exam}\rm\label{I-red}
  Let $\Phi:=\{J_1=(2), ~J_n=(2).(3)^{n-1}, n\geq 2\}$ be an inverse family of ideals of a ring $R:=\mathbb{Z}$ and $M$ be a  $\mathbb{Z}$-module $\mathbb{Z}/6\mathbb{Z}$. For $I:=(2)$, $M$ is $I$-reduced since $(0:_M I)=(0:_M I^2)$ but $M$ is not $\Phi$-reduced.   
\end{exam}
\begin{exam}\rm\label{Phi-red}
 Let $\Phi:=\{J_1=(x), ~J_n=(x(x-1)^{n-1}), n\geq 2\}$ be an inverse family of ideals of a ring $R:=\mathbb{Z}[x]$ and $M$ be  $\mathbb{Z}[x]$-module $\mathbb{Z}[x]/(x)^2$. For $I:=(x)$ and $(0:_M \mathcal{A})=(0:_M \mathcal{L})$ for all $\mathcal{A}\in \Phi$. So, $M$ is $\Phi$-reduced but not $I$-reduced.   
\end{exam}

\section{$\Phi$-coreduced modules}
\begin{paragraph}\noi
    In this section, we dualize the work done in Section $2$. 	The $I$-adic completion functor $\Lambda_I$ associates to each $R$-module $M$, an $R$-module

	\begin{equation}
	 \Lambda_I(M):=\underset{k}{\underleftarrow{\lim}} (M/I^k M)\cong \underset{k}{\underleftarrow{\lim}}(R/I^k\otimes_R M).
	\end{equation}

	The $I$-adic completion functor was crucial in characterizing $I$-coreduced modules in \cite{SsevviiriI2024} and \cite{SsevviiriII2026}. We generalize it using an inverse family $\Phi$. For an $R$-module $M$ and an inverse family $\Phi$, we define

	\begin{equation}
	\Lambda_{\Phi}(M):=\underset{\mathcal{A} \in {\Phi}}{\underleftarrow{\lim}}(M/\mathcal{A}M)\cong \underset{\mathcal{A} \in {\Phi}}{\underleftarrow{\lim}}(R/\mathcal{A}\otimes_R M).
	\end{equation}

\end{paragraph}
\begin{defn}\rm
 Let $\Phi$ be an inverse family of ideals of a ring $R$, having a greatest proper member $\mathcal{L}$ with respect to inclusion. An $R$-module $M$ is said to be $\Phi$-\textit{coreduced} if for every $\mathcal{A} \in \Phi$, $\mathcal{A}M = \mathcal{L}M$.
    
\end{defn}
\begin{exam}
    
\rm For a $\mathbb{Z}$-module $M:=\mathbb{Z}/2\mathbb{Z}$ and an inverse family of ideals of $\mathbb{Z}$; $\Phi:=\{(2), (2^2.3),(2^3.3^2),(2^4.3^3),\cdots\}$  with $\mathcal{L}:=(2)$, $M$ is $\Phi$-coreduced since for all $\mathcal{A}\in \Phi$, $\mathcal{A}M=\mathcal{L}M$.
    \end{exam}
\begin{exam}\rm\label{I-core}
    Let $\Phi:=\{J_1=(2), ~J_n=(2)^{n-1}(3), n\geq 2\}$ be an inverse family of ideals of a ring $R:=\mathbb{Z}$ and a $\mathbb{Z}$-module $M:=\mathbb{Z}/3\mathbb{Z}$. For $I:=(2)$, $M$ is $I$-coreduced since $I^2M=IM$ but $M$ is not $\Phi$-coreduced.
\end{exam}

\begin{exam}\rm \label{phi-core}
    Let $\Phi:=\{J_1=(x), ~J_n=(x(1+x)^{n-1})\}, n\geq 2\}$ be an inverse family of ideals of a ring $R:=\mathbb{Z}[x]$ and a $\mathbb{Z}[x]$-module $M:=\mathbb{Z}[x]/(x)^2$. For $I:=(x)$, $M$ is $\Phi$-coreduced since $\mathcal{A}M=\mathcal{L}M$ for all $\mathcal{A} \in \Phi$ but $M$ is not $I$-coreduced.
\end{exam}

\begin{propn}\label{dualidealsyst1}\rm
 
Let \(M\) be an \(R\)-module. Let \(\Phi\) be an inverse family of ideals of \(R\), and let \(\mathcal L\) be a greatest proper member of \(\Phi\) with respect to inclusion. Then the following statements are equivalent:
\begin{enumerate}
    \item  \(M\) is \(\Phi\)-coreduced;

\item for every \(\mathcal A\in\Phi\), the canonical epimorphism

   $$
   M/\mathcal AM\longrightarrow M/\mathcal LM
   $$

   is an isomorphism;

\item for every \(\mathcal A\in\Phi\), the canonical morphism

   $$
   R/\mathcal A\otimes_R M
   \longrightarrow
   R/\mathcal L\otimes_R M
   $$

   induced by the canonical epimorphism \(R/\mathcal A\to R/\mathcal L\) is an isomorphism.
   \end{enumerate}

If these equivalent conditions hold, then there are natural isomorphisms

$$
\Lambda_\Phi(M)
\cong M/\mathcal LM
\cong R/\mathcal L\otimes_R M.
$$

Consequently, $\mathcal L\Lambda_\Phi(M)=0$.    
\end{propn}

\begin{prf} 
Since \(\mathcal L\) is a greatest member of \(\Phi\), for every \(\mathcal A\in\Phi\) we have $\mathcal A\subseteq\mathcal L$. Consequently, $\mathcal AM\subseteq\mathcal LM$, and hence there is a canonical epimorphism

$$
q_{\mathcal A}:M/\mathcal AM\longrightarrow M/\mathcal LM,
\qquad
m+\mathcal AM\longmapsto m+\mathcal LM.
$$

We prove the equivalences.

\((1)\Rightarrow(2)\). Suppose that \(M\) is \(\Phi\)-coreduced. By definition,
$\mathcal AM=\mathcal LM$ 
for every \(\mathcal A\in\Phi\). Therefore
$M/\mathcal AM=M/\mathcal LM$, and the canonical map \(q_{\mathcal A}\) is an isomorphism.

\((2)\Rightarrow(1)\). Suppose that for every \(\mathcal A\in\Phi\), the canonical epimorphism $q_{\mathcal A}:M/\mathcal AM\longrightarrow M/\mathcal LM$ is an isomorphism. Its kernel is $
\ker(q_{\mathcal A})= \mathcal LM/\mathcal AM$, 
because \(\mathcal AM\subseteq\mathcal LM\). Since \(q_{\mathcal A}\) is an isomorphism, its kernel is zero. Hence $\mathcal LM/\mathcal AM=0$, and therefore $\mathcal AM=\mathcal LM$. This holds for every \(\mathcal A\in\Phi\), so \(M\) is \(\Phi\)-coreduced. Thus \((1)\) and \((2)\) are equivalent. Next, for every ideal \(I\) of \(R\), there is a natural isomorphism $R/I\otimes_R M\cong M/IM$, given by $
(r+I)\otimes m\longmapsto rm+IM$.  Under these natural identifications, the canonical morphism 
$$
R/\mathcal A\otimes_RM
\longrightarrow
R/\mathcal L\otimes_RM
$$

corresponds exactly to the canonical epimorphism

$$
M/\mathcal AM\longrightarrow M/\mathcal LM.
$$

Therefore \((2)\) and \((3)\) are equivalent. It remains to determine the completion. Suppose that the equivalent conditions above hold. Since $\mathcal AM=\mathcal LM$ for every \(\mathcal A\in\Phi\), every term in the inverse system defining \(\Lambda_\Phi(M)\) is canonically equal to $M/\mathcal LM$. 
Moreover, the transition maps become identity maps under these identifications. Thus the inverse system is canonically constant. Consequently,

$$
\Lambda_\Phi(M)
=
\varprojlim_{\mathcal A\in\Phi}M/\mathcal AM
\cong M/\mathcal LM.
$$

Using the standard natural isomorphism $M/\mathcal LM\cong R/\mathcal L\otimes_RM$, we obtain $\Lambda_\Phi(M) \cong M/\mathcal LM
\cong R/\mathcal L\otimes_RM$. Finally, \(\mathcal L\) annihilates \(M/\mathcal LM\). Hence, through the above natural isomorphism, $
\mathcal L\Lambda_\Phi(M)=0$. This completes the proof.  
\end{prf}

\begin{paragraph}\noi
If $I$ is an ideal of $R$ and $\Phi=\{I^n\}_{n\in\mathbb{Z}_+}$, then $\Phi$-coreduced
$R$-modules are precisely the $I$-coreduced modules, and Proposition \ref{dualidealsyst1} reduces
to the classical $I$-coreduced case in Proposition \ref{Icorem}.
\end{paragraph}
\begin{propn}\rm \cite[Proposition 2.3]{SsevviiriI2024}\label{Icorem}
    For any $R$-module $M$ and an ideal $I$ of a ring $R$, the following statements are equivalent:
    \begin{enumerate}
        \item $M$ is $I$-coreduced,
\item $IM = I^2M$,
\item $R/I\otimes_R M \cong R/I^2 \otimes_R M$,
\item $ \Lambda_I(M) \cong R/I \otimes_R M$,
\item $I\Lambda_I(M)=0$.
    \end{enumerate}
    \end{propn}
\begin{propn}\rm
The class of $\Phi$-coreduced $R$-modules is closed under quotient modules, arbitrary direct sums and direct limits.
\end{propn}

\begin{prf}
Let $M$ be $\Phi$-coreduced and let $N\subseteq M$. For every
$\mathcal{A}\in\Phi$, $\mathcal{A}(M/N)=\frac{\mathcal{A}M+N}{N}
           =\frac{\mathcal{L}M+N}{N} =\mathcal{L}(M/N)$. Hence $M/N$ is $\Phi$-coreduced. Now let $\{M_i\}_{i\in J}$ be a family of $\Phi$-coreduced
modules. For every $\mathcal{A}\in\Phi$,
\[
    A\left(\bigoplus_{i\in J}M_i\right)
    =\bigoplus_{i\in J}\mathcal{A}M_i
    =\bigoplus_{i\in J}\mathcal{L}M_i
    =\mathcal{L}\left(\bigoplus_{i\in J}M_i\right).
\]
Thus arbitrary direct sums of $\Phi$-coreduced modules are $\Phi$-coreduced.
Finally, let $\{M_i\}_{i\in J}$ be a directed system of
$\Phi$-coreduced modules and put $    M=\varinjlim_{i\in J}M_i$. For every ideal $\mathcal{A}$ of $R$, multiplication by $\mathcal{A}$ commutes with
direct limits. Hence
\[
    \mathcal{A}M
    =\mathcal{A}\left(\varinjlim_i M_i\right)
    \cong\varinjlim_i \mathcal{A}M_i.
\]
Since each $M_i$ is $\Phi$-coreduced, $ \mathcal{A}M_i=\mathcal{L}M_i$ for every $i$. Therefore
\[
    \mathcal{A}M
    \cong\varinjlim_i \mathcal{A}M_i
    =\varinjlim_i \mathcal{L}M_i
    \cong \mathcal{L}\left(\varinjlim_i M_i\right)
    =\mathcal{L}M.
\]
Thus $M$ is $\Phi$-coreduced.
\end{prf}
 
\begin{propn}\label{dusl to 2.4}
Let $\Phi_1$ and $\Phi_2$ be inverse families of ideals of a
commutative ring $R$, having greatest proper members $\mathcal{L}_1$ and
$\mathcal{L}_2$, respectively, and suppose that $\mathcal{L}_1=\mathcal{L}_2=:\mathcal{L}$.

\begin{enumerate}
        \item Suppose that for every $\mathcal{B}\in\Phi_2$ there exists
    $\mathcal{A}\in\Phi_1$ such that $\mathcal{A}\subseteq \mathcal{B}$. If $M$ is     $\Phi_1$-coreduced, then $M$ is $\Phi_2$-coreduced. Moreover,
    \[
        \Lambda_{\Phi_1}(M)\cong \Lambda_{\Phi_2}(M)
        \cong M/\mathcal{L}M.
    \]

    \item Suppose that for every $\mathcal{A}\in\Phi_1$ there exists
    $\mathcal{B}\in\Phi_2$ such that $\mathcal{B}\subseteq \mathcal{A}$. If $M$ is     $\Phi_2$-coreduced, then $M$ is $\Phi_1$-coreduced. Moreover,
    \[
        \Lambda_{\Phi_1}(M)\cong \Lambda_{\Phi_2}(M)
                    \cong M/\mathcal{L}M.
    \]
\end{enumerate}
\end{propn}

\begin{prf}
We prove (1); the proof of (2) is symmetric. Suppose that $M$ is $\Phi_1$-coreduced. Let $\mathcal{B}\in\Phi_2$ be
arbitrary. By hypothesis, there exists $\mathcal{A}\in\Phi_1$ such that
$\mathcal{A}\subseteq \mathcal{B}$. 
Since $M$ is $\Phi_1$-coreduced, we have
$\mathcal{A}M=\mathcal{L}_1M=\mathcal{L}M$. On the other hand, since $\mathcal{A}\subseteq \mathcal{B}$, we have $\mathcal{A}M\subseteq \mathcal{B}M$. Moreover, because $\mathcal{L}=\mathcal{L}_2$ is the greatest proper member of
    $\Phi_2$, every $\mathcal{B}\in\Phi_2$ satisfies $\mathcal{B}\subseteq \mathcal{L}$, and hence $\mathcal{B}M\subseteq\mathcal{L}M$.
Consequently, $\mathcal{L}M=\mathcal{A}M\subseteq \mathcal{B}M\subseteq \mathcal{L}M$. 
Therefore $\mathcal{B}M=\mathcal{L}M$.  Since $\mathcal{B}\in\Phi_2$ was arbitrary, it follows that $\mathcal{B}M=\mathcal{L}M$  for every $\mathcal{B}\in\Phi_2$.
Thus $M$ is $\Phi_2$-coreduced. Since $M$ is both $\Phi_1$-coreduced and $\Phi_2$-coreduced,
Proposition~\ref{dualidealsyst1} yields natural isomorphisms $\Lambda_{\Phi_1}(M)
        \cong M/\mathcal{L}_1M 
    = M/\mathcal{L}M$
and $\Lambda_{\Phi_2}(M)
    \cong M/\mathcal{L}_2M
    = M/\mathcal{L}M$.  Hence
\[
    \Lambda_{\Phi_1}(M)
    \cong \Lambda_{\Phi_2}(M)
    \cong M/\mathcal{L}M.
\]

\end{prf} 

\begin{propn}\rm \label{dual to systms}
  Let $I$ be an ideal of a ring $R$. Suppose that $\Phi=\{ J_1, J_2, J_3,\cdots\}$ is an inverse family of ideals of a ring $R$  such that $ J_1 \supseteq J_2 \supseteq J_3\supseteq \cdots$ and $I=J_1$. If  $ J_n\subseteq I^n ~(\text{resp}.~ I^n \subseteq J_n)$ for all positive integers $n \geq 2$, then whenever $M$ is  $\Phi$-coreduced (resp. $I$-coreduced),   $M$ is also $I$-coreduced (resp. $\Phi$-coreduced). In both cases, $\Lambda_\Phi(M)\cong\Lambda_I(M)$.  
\end{propn}
\begin{prf}\rm
This is a special case of Proposition \ref{dusl to 2.4}.
\end{prf}
\begin{exam}\rm
   Let $I:=2\mathbb{Z}$  be an ideal of a ring $R:=\mathbb{Z}$. Suppose that $\Phi:=\{J_n=2^n\cdot3^{n-1}|n\geq 1\}$ is an inverse family of ideals of $R$ with $J_1=I$. Clearly $J_n\subseteq I^n$ for $n\geq 2$. Take the $\mathbb{Z}$-module $M=\mathbb{Z}/2\mathbb{Z}$. Since $J_nM=J_1M=IM=I^nM=0$, $M$ is $\Phi$-coreduced and also $I$-coreduced. 
\end{exam}

\subsection{Relations between $\Phi$-reduced and
$\Phi$-coreduced modules}

\begin{propn}\label{red-coreduced-progenerator}\rm
Let $\Phi$ be an inverse family of ideals of a commutative ring $R$ having a greatest proper member $\mathcal L$ with respect to inclusion.
\begin{enumerate}
\item If $M$ is a progenerator and is $\Phi$-coreduced, then $M$ is $\Phi$-reduced.
\item If $M$ is an injective cogenerator and is $\Phi$-reduced, then $M$ is $\Phi$-coreduced.
\end{enumerate}
\end{propn}

\begin{prf}
We prove the two assertions separately.

\medskip
\noindent
\textbf{(1)} Suppose that $M$ is a progenerator and is $\Phi$-coreduced. Let $\mathcal A\in\Phi$. Since $M$ is $\Phi$-coreduced, Proposition~\ref{dualidealsyst1} implies that the canonical morphism 
$$
R/\mathcal A\otimes_R M
\longrightarrow
R/\mathcal L\otimes_R M,
$$

induced by the canonical epimorphism $
q_{\mathcal A}:R/\mathcal A\longrightarrow R/\mathcal L$, 
is an isomorphism. Since $M$ is a progenerator, it is, in particular, faithfully flat. Hence the functor $-\otimes_R M$
reflects isomorphisms. Therefore, since

$$
q_{\mathcal A}\otimes_R 1_M:
R/\mathcal A\otimes_R M
\longrightarrow
R/\mathcal L\otimes_R M
$$

is an isomorphism, the canonical map $q_{\mathcal A}:R/\mathcal A\longrightarrow R/\mathcal L$ itself is an isomorphism. 
Now $\mathcal A\subseteq\mathcal L$, and the kernel of $q_{\mathcal A}$ is
$\ker(q_{\mathcal A})=\mathcal L/\mathcal A$. Thus $\mathcal L/\mathcal A=0$, and hence $\mathcal A=\mathcal L$.  Since $\mathcal A\in\Phi$ was arbitrary, every member of $\Phi$ is equal to $\mathcal L$. 
It follows immediately that, for every $m\in M$ and every $\mathcal A\in\Phi$,
$\mathcal Am=0$ implies that  $\mathcal Lm=0$, because $\mathcal A=\mathcal L$. Therefore $M$ is $\Phi$-reduced.

\medskip
\noindent
\textbf{(2)} Suppose now that $M$ is an injective cogenerator and is $\Phi$-reduced. Let $\mathcal A\in\Phi$. Since $M$ is $\Phi$-reduced, the canonical morphism

$$
\operatorname{Hom}_R(R/\mathcal L,M)
\longrightarrow
\operatorname{Hom}_R(R/\mathcal A,M),
$$

induced by the canonical epimorphism $q_{\mathcal A}:R/\mathcal A\longrightarrow R/\mathcal L$,  is an isomorphism. Since $M$ is an injective cogenerator, the contravariant functor  $\operatorname{Hom}_R(-,M)$ is exact and faithful, and hence reflects isomorphisms. Therefore the canonical morphism $q_{\mathcal A}:R/\mathcal A\longrightarrow R/\mathcal L$ is an isomorphism. As above, $\ker(q_{\mathcal A})=\mathcal L/\mathcal A$.  Hence $\mathcal L/\mathcal A=0$, so $\mathcal A=\mathcal L$.  Since $\mathcal A\in\Phi$ was arbitrary, every member of $\Phi$ is equal to $\mathcal L$. Consequently, $\mathcal AM=\mathcal LM$  for every $\mathcal A\in\Phi$. Thus $M$ is $\Phi$-coreduced. 
\end{prf}

\begin{propn}\rm\label{Hom-reversed1}
Let $\Phi$ be an inverse family of ideals of a commutative ring
$R$ having greatest proper member $L$, and let $M$ be an
$R$-module.

\begin{enumerate}
    \item If $M$ is $\Phi$-coreduced, then
    $\operatorname{Hom}_R(M,K)$ is $\Phi$-reduced for every
    $R$-module $K$.

    \item Conversely, if $K$ is an injective cogenerator and
    $\operatorname{Hom}_R(M,K)$ is $\Phi$-reduced, then $M$ is
    $\Phi$-coreduced.
\end{enumerate}
\end{propn}

\begin{prf} 
Put $    H=\operatorname{Hom}_R(M,K)$. For every ideal $\mathcal{A}$ of $R$, we have
$(0:_H \mathcal{A})     =     \{f\in\operatorname{Hom}_R(M,K)\mid f(\mathcal{A}M)=0\}$. Indeed, for $f\in H$,
$    \mathcal{}Af=0$ if and only if $f(am)=0$ for all $a\in\mathcal{A},\ m\in M$, which is equivalent to $f(\mathcal{A}M)=0$. Suppose first that $M$ is $\Phi$-coreduced. Then, for every
$\mathcal{A}\in\Phi$, $\mathcal{A}M=\mathcal{L}M$. Consequently,
\[
\begin{aligned}
    (0:_H \mathcal{A})
    &=
    \{f\in H\mid f(\mathcal{A}M)=0\}\\
    &=
    \{f\in H\mid f(\mathcal{L}M)=0\}\\
    &=
    (0:_H \mathcal{L}).
\end{aligned}
\]
By Proposition~\ref{p2}, $H$ is $\Phi$-reduced. Conversely, suppose that $K$ is an injective cogenerator and that $H=\operatorname{Hom}_R(M,K)$ is $\Phi$-reduced. We prove that $\mathcal{A}M=\mathcal{L}M$ for every $\mathcal{A}\in\Phi$. Fix $\mathcal{A}\in\Phi$. Since $\mathcal{A}\subseteq \mathcal{L}$, we always have
$\mathcal{A}M\subseteq \mathcal{L}M$. Suppose, for contradiction, that
$\mathcal{A}M\subsetneq \mathcal{L}M$. Then $\mathcal{L}M/\mathcal{A}M\neq0$. Since $K$ is a cogenerator, there exists a nonzero homomorphism $g:\mathcal{L}M/\mathcal{A}M\longrightarrow K$. Now $\mathcal{L}M/\mathcal{A}M$ is naturally a submodule of $M/\mathcal{A}M$. Since $K$ is
injective, $g$ extends to a homomorphism
$   \widetilde g:M/\mathcal{A}M\longrightarrow K$. Let
$ \pi:M\longrightarrow M/\mathcal{A}M$ be the canonical projection and define $f=\widetilde g\circ\pi:M\longrightarrow K$. Since $f(\mathcal{A}M)=0$, we have $f\in(0:_H \mathcal{A})$. On the other hand, the restriction of $f$ to $\mathcal{L}M$ induces the
nonzero map $g$, and hence $ f(\mathcal{L}M)\neq0$. Therefore
$f\notin(0:_H \mathcal{L})$. Thus $  (0:_H\mathcal{A})\neq(0:_H\mathcal{L})$, 
contradicting the assumption that $H$ is $\Phi$-reduced.
Hence $\mathcal{A}M=\mathcal{L}M$ for every $\mathcal{A}\in\Phi$, and therefore $M$ is $\Phi$-coreduced. 
\end{prf}

\begin{rem}\rm 
The notions of $\Phi$-reducedness and $\Phi$-coreducedness are
dual stabilization conditions. Namely, $M$ is $\Phi$-reduced
precisely when  $ (0:_M \mathcal{A})=(0:_M \mathcal{L})$ 
   for every  $\mathcal{A}\in\Phi$, whereas $M$ is $\Phi$-coreduced precisely when
$\mathcal{A}M=\mathcal{L}M$ for every $\mathcal{A}\in\Phi$. Thus the former is an annihilator stabilization condition and the latter is an image stabilization condition. Proposition~\ref{Hom-reversed1} shows that these two conditions are interchanged by
$\operatorname{Hom}_R(-,K)$ when $K$ is an injective cogenerator.
\end{rem}

\section{A GM-type adjunction $\&$ an MGM-type phenomenon}

\begin{paragraph}\noi
 For an $R$-module $M$, the canonical $\Phi$-adic completion morphism is

$$
\eta_M:M\longrightarrow\Lambda_\Phi(M)
=\varprojlim_{\mathcal A\in\Phi}M/\mathcal AM,
\qquad
m\longmapsto (m+\mathcal AM)_{\mathcal A\in\Phi}.
$$

The module $M$ is said to be \emph{$\Phi$-adically complete} if $\eta_M$ is an isomorphism. An $R$-module $M$ is said to be \emph{$\Phi$-torsion} if $\Gamma_\Phi(M)=M$, see \cite{Brodmann2013}. Let $(R\text{-Mod})_{\Phi\text{-red}}$ (resp. $(R\text{-Mod})_{\Phi\text{-cor}}$) denote the full subcategory of $R$-Mod consisting of $\Phi$-reduced (resp. $\Phi$-coreduced) $R$-modules.
\end{paragraph}

\begin{thm}\label{GGM-Duality}{\rm\bf[GM-type  adjunction]}
	Let $\Phi$ be an inverse family of ideals of $R$ with $\mathcal{L}$ as a greatest proper member in $\Phi$. The functors
	
\begin{equation}\label{systitorsn}
	\Gamma_{\Phi}: (R\text{-Mod})_{\Phi\text{-red}} \longrightarrow (R\text{-Mod})_{\Phi\text{-cor}}
	\end{equation}
	 and
\begin{equation}\label{systadic}		
	\Lambda_{\Phi} : (R\text{-Mod})_{\Phi\text{-cor}} \longrightarrow (R\text{-Mod})_{\Phi\text{-red}}
\end{equation}	
exist and  form an adjoint pair such that for any 
	 $N \in (R\text{-Mod})_{\Phi\text{-red}}$ and $M \in (R\text{-Mod})_{\Phi\text{-cor}}$, we have
	$$
	\text{Hom}_R(\Lambda_{\Phi}(M), N) \cong \text{Hom}_R(M, \Gamma_{\Phi}(N))$$ naturally in \(M\) and $N$.
\end{thm}
\begin{prf}
	Let $N \in (R\text{-Mod})_{\Phi\text{-red}}$. By Proposition \ref{p2}, $\Gamma_{\Phi}(N) \cong \text{Hom}_R(R/\mathcal{L},N)
	$. It follows that $\mathcal{L}\Gamma_{\Phi}(N) = 0$ so that $\mathcal{A}\Gamma_{\Phi}(N) =\mathcal{L}\Gamma_{\Phi}(N) = 0$ and $\Gamma_{\Phi}(N)\in (R\text{-Mod})_{\Phi\text{-cor}}$. This establishes  Functor (\ref{systitorsn}). Now, let $M \in (R\text{-Mod})_{\Phi\text{-cor}}$. By Proposition \ref{dualidealsyst1}, $\Lambda_{\Phi}(M)\cong R/\mathcal{L}\otimes_R M$. So, $\mathcal{L}\Lambda_{\Phi}(M)=0$ and $\mathcal{A}\Lambda_{\Phi}(M)=0$. This implies that $(0:_{\Lambda_{\Phi}(M)} \mathcal{L})=(0:_{\Lambda_{\Phi}(M)} \mathcal{A})=\Lambda_{\Phi}(M)$. By Proposition \ref{p2}, $\Lambda_{\Phi}(M)\in (R\text{-Mod})_{\Phi\text{-red}}$ which establishes Functor (\ref{systadic}). The two functors (\ref{systitorsn}) and (\ref{systadic}) are adjoint since for $N \in (R\text{-Mod})_{\Phi\text{-red}}, \Gamma_{\Phi}(N)\cong \text{Hom}_R (R/\mathcal{L}, N)$ and for $M \in (R\text{-Mod})_{\Phi\text{-cor}}, ~\Lambda_{\Phi}(M)\cong R/\mathcal{L}\otimes_R M$. Moreover, $R/\mathcal{L}\otimes-$ is the left adjoint of $\text{Hom}_R(R/\mathcal{L},-)$.
\end{prf}

\begin{thm}\label{MGM-equivalence}{\rm\bf [MGM-type characterization]}
Let $\Phi$ be an inverse family of ideals of a ring $R$ having a greatest proper member $\mathcal L$ with respect to inclusion.   Then the following statements are equivalent:
\begin{enumerate}
\item $M$ is both $\Phi$-torsion and $\Phi$-reduced;
\item $M$ is both $\Phi$-adically complete and $\Phi$-coreduced;
\item $\mathcal LM=0$.
\end{enumerate}
Moreover, under any of these equivalent conditions,
$\Gamma_\Phi(M)=M$ and $\Lambda_\Phi(M)\cong M$ naturally.
\end{thm}

\begin{prf}
We prove $(1)\Longleftrightarrow(3)\Longleftrightarrow(2)$.
\medskip
\noindent
\textbf{$(1)\Rightarrow(3)$.}
Suppose that $M$ is both $\Phi$-torsion and $\Phi$-reduced. Since $M$ is $\Phi$-torsion, $\Gamma_\Phi(M)=M$. Since $M$ is $\Phi$-reduced, Proposition~\ref{p2} gives $\Gamma_\Phi(M)=(0:_M\mathcal L)$. Consequently, $M=(0:_M\mathcal L)$, and hence
$\mathcal LM=0$. 
\medskip
\noindent
\textbf{$(3)\Rightarrow(1)$.}
Suppose that $\mathcal LM=0$. Since $\mathcal A\subseteq\mathcal L$ for every $\mathcal A\in\Phi$, we have $\mathcal AM=0$ for every $\mathcal A\in\Phi$. Thus, for every $m\in M$ and every $\mathcal A\in\Phi$, $\mathcal Am=0 \quad\Longrightarrow\quad \mathcal Lm=0$. 
Therefore $M$ is $\Phi$-reduced. Moreover, since $\mathcal L\in\Phi$ and $\mathcal LM=0$, every $m\in M$ belongs to $(0:_M\mathcal L)\subseteq\Gamma_\Phi(M)$. 
Hence $M\subseteq\Gamma_\Phi(M)\subseteq M$, and therefore $\Gamma_\Phi(M)=M$. 
Thus $M$ is $\Phi$-torsion. Hence (1) holds.

\medskip
\noindent
\textbf{$(2)\Rightarrow(3)$.}
Suppose that $M$ is both $\Phi$-adically complete and $\Phi$-coreduced. Since $M$ is $\Phi$-coreduced, Proposition~\ref{dualidealsyst1} gives a natural isomorphism $\Lambda_\Phi(M)\cong M/\mathcal LM$.
Under this natural identification, the canonical completion morphism

$$
\eta_M:M\longrightarrow\Lambda_\Phi(M)
$$

corresponds to the canonical quotient morphism

$$
q:M\longrightarrow M/\mathcal LM,
\qquad
m\longmapsto m+\mathcal LM.
$$

Since $M$ is $\Phi$-adically complete, $\eta_M$ is an isomorphism. Hence the quotient map $q$ is an isomorphism. In particular, $\ker(q)=\mathcal LM=0$. Therefore $\mathcal LM=0$. 
\medskip
\noindent
\textbf{$(3)\Rightarrow(2)$.}
Suppose that $\mathcal LM=0$. Since $\mathcal A\subseteq\mathcal L$ for every $\mathcal A\in\Phi$, it follows that $\mathcal AM=0$ 
for every $\mathcal A\in\Phi$. Hence $\mathcal AM=\mathcal LM=0$ for every $\mathcal A\in\Phi$, so $M$ is $\Phi$-coreduced. Furthermore, $M/\mathcal AM=M$ for every $\mathcal A\in\Phi$. Thus the inverse system $\left\{M/\mathcal AM\right\}_{\mathcal A\in\Phi}$ is canonically the constant inverse system with value $M$ and identity transition maps. Consequently, $\Lambda_\Phi(M)
=
\varprojlim_{\mathcal A\in\Phi}M/\mathcal AM \cong M$. Under this identification, the canonical completion morphism 
$\eta_M:M\longrightarrow\Lambda_\Phi(M)$ is the identity morphism. Hence $\eta_M$ is an isomorphism, and therefore $M$ is $\Phi$-adically complete. Thus (2) holds.  
\end{prf}

\section{The torsion theory induced by $\Phi$-reduced and $\Phi$-coreduced modules}

\begin{paragraph}\noi
 Throughout this section, let $\Phi$ be a system of ideals of a ring $R$.  For an abelian category $\mathscr{C}$, a \textit{torsion theory} $\tau$ is a pair $(\mathscr{T}, \mathscr{F})$ of classes of objects of $\mathscr{C}$ such that:
\begin{enumerate}
\item $\operatorname{Hom}(T, F) = 0$ for all $T \in \mathscr{T}$, $F \in \mathscr{F}$;
\item if $\operatorname{Hom}(A, F) = 0$ for all $F \in \mathscr{F}$, then $A \in \mathscr{T}$;
\item if $\operatorname{Hom}(T, B) = 0$ for all $T \in \mathscr{T}$, then $B \in \mathscr{F}$.
\end{enumerate}
The class $\mathscr{T}$ is called the \emph{torsion class} and its objects are called \emph{torsion objects}; the class $\mathscr{F}$ is called the \emph{torsionfree class} and its objects are called \emph{torsionfree objects}. A torsion theory $(\mathscr{T}, \mathscr{F})$ is called \emph{hereditary} if its torsion class $\mathscr{T}$ is closed under sub-objects.
\end{paragraph}
\begin{defn}\rm{\cite[Definition 2.1.10]{Brodmann2013}}\label{systofidls}
 Let $(\Lambda,\leq)$ be a non-empty directed partially ordered set.  A {\it  system of ideals} of a ring $R$ over $\Lambda$ is  an  inverse family of ideals of $R$, denoted by $\{I_i\}_{i \in \Lambda}$ with an additional   property that for every $\alpha,\beta \in \Lambda$, there exists $\delta \in \Lambda$ with $\delta\geq\alpha$ and $\delta\geq\beta$  such that $I_{\delta}\subseteq I_ {\alpha}I_ {\beta}$.
\end{defn}

\begin{paragraph}\noi
    The additional multiplicative property of a system of ideals allows one to relate products of ideals to members of an inverse family, a feature that is indispensable in establishing the radical and torsion theory properties developed in this section. Systems of ideals have played a fundamental role in the development of local cohomology theories. In \cite{Brodmann2013},  a system of ideals $\Phi$ was required both for the radicality of the functor $\Gamma_\Phi$ on $R\text{-Mod}$ with $R$ Noetherian, see \cite[Exercise 2.1.13(1)]{Brodmann2013} and for the right derived functor of $\Gamma_\Phi$ to mimic the theory of local cohomology completely; the analogue of \cite[Exercise 2.1.4]{Brodmann2013} fails unless $\Phi$ is a system of ideals.
\end{paragraph}

\begin{exam}\rm 
Let \(R\) be a principal ideal domain, \(a\in R\) a non‑unit, and let \(f:\mathbb{N}\to\mathbb{N}\) be an increasing function such that \(f(m+n)\ge f(m)+f(n)\) for all \(m,n\in\mathbb{N}\) (e.g., \(f(n)=n\), \(f(n)=cn\) with \(c\ge1\), or \(f(n)=n^2\)). Set \(\Lambda = \mathbb{N}\) with the usual order and define $I_n := (a^{f(n)})$.
For any \(\alpha,\beta\in\mathbb{N}\) take \(\delta = \alpha+\beta\). Then
$I_\delta = (a^{f(\alpha+\beta)}) \subseteq (a^{f(\alpha)+f(\beta)}) = I_\alpha I_\beta$, and because \(f\) is increasing, \(I_n\) is decreasing. Thus, \(\{I_n\}_{n\in\mathbb{N}}\) is a system of ideals of $R$.
\end{exam}

\begin{lem}\rm\label{lem1}
  Let $\Phi$ be a system of ideals of a ring $R$ having a greatest proper member $\mathcal{L}$ with respect to inclusion.\begin{enumerate}
      \item If $M$ is a $\Phi$-reduced $R$-module, then $M$ is $\mathcal{L}$-reduced.
      \item If $M$ is a $\Phi$-coreduced $R$-module then, $M$ is $\mathcal{L}$-coreduced.
      
  \end{enumerate} 
\end{lem}
\begin{prf}\rm
\begin{enumerate}
    \item Suppose that $\mathcal{L}^2m=0$ for some $m\in M$. Since $\Phi$ is a system of ideals of a ring $R$ and $\mathcal{L}\in \Phi$, there exists an ideal $\mathcal{A}\in \Phi$ such that $\mathcal{A}\subseteq \mathcal{L}^2$. This implies that $\mathcal{A}m\subseteq \mathcal{L}^2m=0$. So, $\mathcal{A}m=0$. Since $M$ is $\Phi$-reduced,  $\mathcal{L}m=0$. Hence, $M$ is $\mathcal{L}$-reduced.

 \item Since $\Phi$ is a system of ideals of a ring $R$ and $\mathcal{L}\in \Phi$, there exists an ideal $\mathcal{A}\in \Phi$ such that $\mathcal{A}\subseteq \mathcal{L}^2$. From the fact that $M$ is $\Phi$-coreduced, $\mathcal{A}M=\mathcal{L}M$. Hence, $\mathcal{L}M=\mathcal{A}M\subseteq\mathcal{L}^2M\subseteq\mathcal{L}M$. Therefore, $M$ is $\mathcal{L}$-coreduced.
  \end{enumerate}
\end{prf}
\begin{exam}\rm
   Let $R$ be  the  polynomial ring $k[x, y]$  over a field $k$ and $\Phi$  be an inverse family $\{ (x, y^{n}) ~:~n\geq 1\}$ of ideals of $R$ with $\mathcal{L}=(x,y)$. Notice that this inverse family is not a system of ideals. Let $M=R/(x-y,x^2)$ be a $k[x,y]$-module. Since $\mathcal{A}M=\mathcal{L}M$ for every $\mathcal{A}\in \Phi$, $M$ is $\Phi$-coreduced and $\mathcal{L}^2M\neq \mathcal{L}M$ implies that $M$ is not $\mathcal{L}$-coreduced. Also for a $k[x,y]$-module $R/(x^2,xy,y^2)$, $(0:_M \mathcal{A})=(0:_M \mathcal{L})=(x,y)M$ but $\mathcal{L}^2M=0$ while $\mathcal{L}M\neq 0$. Hence $M$ is $\Phi$-reduced but not $\mathcal{L}$-reduced.
\end{exam}
\begin{defn}\rm
A functor $\sigma:R\text{-Mod}\longrightarrow R\text{-Mod}$, $M\mapsto \sigma(M)$, where $\sigma(M)$ is a submodule of $M$, is a:
\begin{enumerate}
    \item \textit{preradical} if for every $R$-module homomorphism $f:A\rightarrow B$, $f(\sigma(A))\subseteq \sigma(B)$.
    \item \textit{radical} if it is a preradical and for any $M\in R\text{-Mod}$, $\sigma(M/\sigma(M))=0$.
\end{enumerate}
    
\end{defn}

\begin{defn}\rm 
     A class $\mathcal{S}$ of $R$-modules is called a \textit{Serre subcategory} of the category of $R$-modules if it is closed under taking submodules, quotients, and extensions. 
     \begin{paragraph}\noi
      Serre subcategories are fundamental in homological algebra. They allow the construction of quotient categories and are closely related to torsion theories.   
     \end{paragraph}
\end{defn}
\begin{propn}\rm\label{radicality}
    Let $\Phi$ be a system of ideals of a ring $R$ over $\Lambda$ and let $\mathcal{L}\in\Phi$ be a greatest proper member of $\Phi$ with respect to inclusion. If $\mathcal{A}_\Phi$ is a Serre subcategory of $R\text{-Mod}$ consisting of $\Phi$-reduced $R$-modules, then the functor $\Gamma_\Phi:\mathcal{A}_\Phi \rightarrow \mathcal{A}_\Phi$ is a left exact radical.
     \end{propn}
    \begin{prf}
    For every $M\in \mathcal{A}_\Phi$, $\Gamma_\Phi(M)$ is a submodule of $M$. Since the Serre subcategory $\mathcal{A}_\Phi$ is closed under submodules, $\Gamma_\Phi(M)\in \mathcal{A}_\Phi$. Thus the functor  $\Gamma_\Phi:\mathcal{A}_\Phi \rightarrow \mathcal{A}_\Phi  $ exists. Since by Proposition \ref{p2}(5), $\Gamma_\Phi(-)=(0:_- \mathcal{L})\cong \text{Hom}_R(R/\mathcal{L},-)$ where $\text{Hom}_R(R/\mathcal{L},-)$ is a left exact functor, the functor $\Gamma_\Phi$ is left exact. 
Let $M$ and $N$ be $\Phi$-reduced modules in $\mathcal{A}_\Phi$  and $f:M\longrightarrow N$ be an $R$-homomorphism. By Proposition \ref{p2}(5), $M$ $\Phi$-reduced implies that $\Gamma_\Phi(M)=(0:_M \mathcal{L})$. We first show that $f(0:_M \mathcal{L})\subseteq (0:_N \mathcal{L})$. Take $m\in (0:_M \mathcal{L})$ then, $\mathcal{L}m=0$ and $\mathcal{L}f(m)=f(\mathcal{L}m)=0$. Hence $f(m)\in (0:_N \mathcal{L})$ and $f(0:_M \mathcal{L})\subseteq (0:_N \mathcal{L})$. Therefore, the functor $\Gamma_\Phi$ is a preradical. Now let $\Bar{m}\in \Gamma_\Phi\left( M/\Gamma_\Phi(M)\right)$. By Definition of the submodule $\Gamma_\Phi(M)$, there exists $I_\alpha \in \Phi$ such that $I_\alpha \bar{m}=\bar{0}$. This implies that $I_\alpha m \subseteq \Gamma_\Phi(M)$. For $M \in \mathcal{A}_\Phi$ is $\Phi$-reduced, by Proposition \ref{p2}, $\Gamma_\Phi(M)=(0:_M \mathcal{L})$. Then, $I_\alpha m \subseteq (0:_M \mathcal{L})$ so that $\mathcal{L}I_\alpha m=0$. Since $\Phi$ is a system of ideals of $R$ over $\Lambda$, there exists $\delta \in \Lambda$ such that $I_\delta \subseteq \mathcal{L}I_\alpha$. So, $I_\delta m=0$ and $m\in \Gamma_\Phi(M)$. Hence, $\Bar{m}=0$. Therefore, the functor $\Gamma_\Phi$ restricted to $\mathcal{A}_\Phi$ is a radical.
\end{prf}

\begin{paragraph}\noi
Example \ref{ex-prop} shows that if we replace the assumption that $\Phi$ is a system of ideals in Proposition \ref{radicality} with an inverse family of ideals, the functor need not be a radical.
\end{paragraph}

\begin{exam}\rm \label{ex-prop}
Let $R:=k[x,y]$ be a polynomial ring over a field $k$ and let $M:= k[x,y]/(x^2,y-1)$ be an $R$-module. Define $\Phi:=\{(xy^{n-1})~:~{n\geq 1}\}$ to be an inverse family of ideals of $R$. $\Phi$ is not a system of ideals. Since $(x^2=0, y=1) \in M$, elements in $M$ are of the form $a_0+a_1x$, with $a_0,a_1\in k$ and $\Gamma_\Phi(M)=(0:_M (x))=a_1(x)$. Also, $M/\Gamma_\Phi(M)\cong k$. Therefore $\Gamma_\Phi(M/\Gamma_\Phi(M))=k\neq 0$. 
    \end{exam}
\begin{paragraph}\noi
  Recall that, in \cite{Jans1965}, Jans showed that  a map $I\mapsto T_I$ defines a one to one correspondence between  an idempotent ideal  $I$ of a ring $R$ and the torsion-torsion free (TTF) class $T_I$ given by $T_I:=\{M\in R\text{-Mod}~|~IM=0\}$. This is called Jans' correspondence; see \cite[Corollary 2.2]{Jans1965}.  Jans' correspondence was generalized in \cite[Theorem 3.2]{SsevviiriII2026}, in which torsion theories associated to a fixed ideal $I$ that is flat as an $R$-module were developed on Serre subcategories consisting of $I$-reduced and $I$-coreduced $R$-modules. Theorem \ref{radicality} generalizes these results with respect to a fixed ideal $I$ to an arbitrary system of ideals $\Phi$ and develops torsion theories on Serre subcategories consisting of $\Phi$-reduced and $\Phi$-coreduced $R$-modules.
\end{paragraph}

\begin{thm}\label{TTF-theorem}
Let \(\Phi\) be a system of ideals of a ring \(R\), and let \(\mathcal L\in\Phi\) be a greatest proper member of \(\Phi\) with respect to inclusion. Let \(\mathcal A_\Phi\) and \(\mathcal B_\Phi\) be Serre subcategories of \(R\)-Mod consisting of \(\Phi\)-reduced and \(\Phi\)-coreduced \(R\)-modules, respectively.

Assume that  $$ \{M\in\mathcal A_\Phi\mid \mathcal LM=0\}
= \{M\in\mathcal B_\Phi\mid \mathcal LM=0\},$$

and denote this common class by $\mathscr T_\Phi$. Define

$$
\mathscr F_\Phi
=
\{M\in\mathcal A_\Phi\mid (0:_M\mathcal L)=0\}
$$

and

$$
\mathfrak C_\Phi
=
\{M\in\mathcal B_\Phi\mid \mathcal LM=M\}.
$$

Then:
\begin{enumerate}
    \item the pair \((\mathscr T_\Phi,\mathscr F_\Phi)\) forms a hereditary torsion theory in the Serre subcategory \(\mathcal A_\Phi\);
    \item   the pair \((\mathfrak C_\Phi,\mathscr T_\Phi)\) forms a torsion theory in the Serre subcategory \(\mathcal B_\Phi\);
    \item  consequently, \(\mathscr T_\Phi\) is simultaneously a torsion class in the torsion theory \((\mathscr T_\Phi,\mathscr F_\Phi)\) on \(\mathcal A_\Phi\) and a torsionfree class in the torsion theory \((\mathfrak C_\Phi,\mathscr T_\Phi)\) on \(\mathcal B_\Phi\).
\end{enumerate}    
\end{thm}

\begin{prf} 
  \begin{enumerate}
      \item      By Proposition \ref{radicality}, the functor   $\Gamma_\Phi:\mathcal A_\Phi\longrightarrow\mathcal A_\Phi$ is a left exact radical. Therefore, by the standard correspondence between left exact radicals and hereditary torsion theories, it induces a hereditary torsion theory on \(\mathcal A_\Phi\). We identify its torsion and torsion-free classes. For every \(M\in\mathcal A_\Phi\), the module \(M\) is \(\Phi\)-reduced. Hence $\Gamma_\Phi(M)=(0:_M\mathcal L).$ 
Therefore $\Gamma_\Phi(M)=M$ if and only if $
(0:_M\mathcal L)=M$, which is equivalent to $
\mathcal LM=0$. Thus the torsion class associated with \(\Gamma_\Phi\) is precisely $\mathscr T_\Phi$. 
Similarly, $\Gamma_\Phi(M)=0$ if and only if $(0:_M\mathcal L)=0$. 
Hence the corresponding torsion-free class is precisely $\mathscr F_\Phi$.  It follows that $(\mathscr T_\Phi,\mathscr F_\Phi)$ 
is a hereditary torsion theory in \(\mathcal A_\Phi\). We prove that  \(\mathscr T_\Phi\) is closed under extensions. Suppose that
$$
0\longrightarrow M_1\longrightarrow M_2\longrightarrow M_3\longrightarrow0
$$

is a short exact sequence in \(\mathcal A_\Phi\), with \(M_1,M_3\in\mathscr T_\Phi\). Thus
$\mathcal LM_1=0$ and $\mathcal LM_3=0$.  For \(m\in M_2\), the image of \(\mathcal Lm\) in \(M_3\) is zero because \(\mathcal LM_3=0\). Hence $\mathcal Lm\subseteq M_1$. 
Since \(\mathcal LM_1=0\), it follows that $\mathcal L^2m=0$. 
Now \(M_2\) is \(\Phi\)-reduced, and since \(\Phi\) is a system of ideals, Lemma \ref{lem1} implies that \(M_2\) is \(\mathcal L\)-reduced. Therefore $\mathcal L^2m=0$ implies that $\mathcal Lm=0$. 
Thus $\mathcal LM_2=0$, so \(M_2\in\mathscr T_\Phi\). This proves extension closure. 

 \item  We now show that $(\mathfrak C_\Phi,\mathscr T_\Phi)$ 
is a torsion theory in \(\mathcal B_\Phi\). First, let \(X\in\mathfrak C_\Phi\) and \(Y\in\mathscr T_\Phi\). Thus $\mathcal LX=X$  and $\mathcal LY=0$. If $ f:X\longrightarrow Y$ 
is an \(R\)-module homomorphism, then $f(X)=f(\mathcal LX)=\mathcal Lf(X)\subseteq\mathcal LY=0$. Hence \(f=0\). Therefore $\operatorname{Hom}_R(\mathfrak C_\Phi,\mathscr T_\Phi)=0$. 
Next, let \(M\in\mathcal B_\Phi\). Since \(M\) is \(\Phi\)-coreduced and \(\Phi\) is a system of ideals, Lemma \ref{lem1} implies that \(M\) is \(\mathcal L\)-coreduced. Consequently, 
$\mathcal L^2M=\mathcal LM$. Consider the canonical short exact sequence

$$
0\longrightarrow\mathcal LM
\longrightarrow M
\longrightarrow M/\mathcal LM
\longrightarrow0.
$$

Since \(\mathcal B_\Phi\) is a Serre subcategory and \(M\in\mathcal B_\Phi\), both \(\mathcal LM\) and \(M/\mathcal LM\) belong to \(\mathcal B_\Phi\). Furthermore, $\mathcal L(\mathcal LM) = \mathcal L^2M
= \mathcal LM$.  Hence $\mathcal LM\in\mathfrak C_\Phi$. 
On the other hand, $\mathcal L(M/\mathcal LM)=0$. 
Therefore $M/\mathcal LM \in
\{N\in\mathcal B_\Phi\mid\mathcal LN=0\}$. 
By the hypothesis identifying the \(\mathcal L\)-annihilated objects of \(\mathcal A_\Phi\) and \(\mathcal B_\Phi\), this class is \(\mathscr T_\Phi\). Thus $M/\mathcal LM\in\mathscr T_\Phi$. 
Consequently every \(M\in\mathcal B_\Phi\) fits into a short exact sequence

$$
0\longrightarrow T_M
\longrightarrow M
\longrightarrow F_M
\longrightarrow0,
$$

where $ T_M=\mathcal LM\in\mathfrak C_\Phi$  and $F_M=M/\mathcal LM\in\mathscr T_\Phi$.  Together with $\operatorname{Hom}_R(\mathfrak C_\Phi,\mathscr T_\Phi)=0$, this proves that $(\mathfrak C_\Phi,\mathscr T_\Phi)$ is a torsion theory in \(\mathcal B_\Phi\). For completeness, the extension closure of \(\mathfrak C_\Phi\) is also proved. Suppose that  $0\longrightarrow X\longrightarrow Y\longrightarrow Z\longrightarrow0 $ 
is exact in \(\mathcal B_\Phi\), with $\mathcal LX=X$ and $\mathcal LZ=Z$.
Let \(y\in Y\), and let \(\bar y\) denote its image in \(Z\). Since \(\mathcal LZ=Z\), we can write $\bar y=l_1z_1+\cdots+l_nz_n$ for some \(l_i\in\mathcal L\) and \(z_i\in Z\). Choose \(y_i\in Y\) mapping to \(z_i\). Then $y-\sum_i l_iy_i$ 
lies in \(X\). Since \(X=\mathcal LX\), this element belongs to \(\mathcal LY\). Also, $\sum_i l_iy_i\in\mathcal LY$. Hence \(y\in\mathcal LY\). Since \(y\) was arbitrary, $Y=\mathcal LY$. Therefore \(Y\in\mathfrak C_\Phi\), proving that \(\mathfrak C_\Phi\) is closed under extensions.

\item  Part (1) shows that 
$\mathscr T_\Phi$ is the torsion class of the hereditary torsion theory $(\mathscr T_\Phi,\mathscr F_\Phi)$ in \(\mathcal A_\Phi\). Part (2) shows that the same class $\mathscr T_\Phi$ is the torsion-free class of the torsion theory $(\mathfrak C_\Phi,\mathscr T_\Phi)$ in \(\mathcal B_\Phi\). Thus \(\mathscr T_\Phi\) occurs simultaneously as a torsion class in one torsion theory and as a torsion-free class in the other. This establishes the desired torsion–torsionfree phenomenon.
 \end{enumerate}
\end{prf} 
\begin{exam}\rm
 Let  $R=k\times k\times k[x,y]$ be a ring  where $k$ is a field and $A_n=\Phi=\{0\times k\times(x^n,y^n)~|~n\geq 1\}$ be a system of ideals of $R$ with $\mathcal{L}=0 \times k\times (x,y)$ a greatest proper member in $\Phi$. Similarly, $R\text{-Mod}\cong  k\text{-vect}\times k\text{-vect}\times k[x,y]\text{-Mod}$. Define $\mathcal{A}_\Phi=\mathcal{B}_\Phi:=k\text{-vect}\times k\text{-vect}\times 0$. This is a  Serre subcategory of $R$-Mod. Objects in $\mathcal{A}_\Phi=\mathcal{B}_\Phi$ are of the form $M=(U, V, 0)$ where $U$ and $V$ are $k$-vector spaces. The ideal $\mathcal{L}$ acts on $M$ as, $\mathcal{L}M=(0, V, 0)$. Similarly for every $n\geq 1$, $(A_n)M=(0, V, 0)=\mathcal{L}M$. Hence every object of $\mathcal{A}_\Phi=\mathcal{B}_\Phi$ is both $\Phi$-reduced and $\Phi$-coreduced. The three classes in Theorem \ref{TTF-theorem} are as follows. First, $\mathscr{T}_\Phi=\{M\in \mathcal{A}_\Phi~|~ \mathcal{L}M=0\}$. For $M=(U, V, 0)$, $\mathcal{L}M=(0, V, 0)$. Thus $\mathcal{L}M=0$ if and only if $V=0$. Hence, $\mathscr{T}_\Phi=\{(U, 0, 0)~|~ U \in k\text{-vect}\}$. Second, an element $(u, v, 0)$ of $M$ is annihilated by $\mathcal{L}$ precisely when $V=0$. Therefore, $(0:_M \mathcal{L})=(U, 0, 0)$. Consequently, $(0:_M \mathcal{L})=0$ if and only if $U=0$. Hence, $\mathscr{F}_\Phi=\{(0, V, 0)~|~ V\in k\text{-vect}\}$. The first torsion theory is $(\mathscr{T}_\Phi,\mathscr{F}_\Phi)=(\{(U, 0, 0)\},\{(0, V, 0)\})$. By Theorem \ref{TTF-theorem} this is a hereditary torsion theory. Next, $\mathfrak{C}_\Phi=\{M\in \mathcal{B}_\Phi~|~ \mathcal{L}M=M\}$. Since $\mathcal{L}M=(0, V, 0)$, we have $\mathcal{L}M=M$ if and only if $U=0$. Hence, $\mathfrak{C}_\Phi=\{(0, V,0)~|~ V \in k\text{-vect}\}$. Therefore, the second torsion theory is $(\mathfrak{C}_\Phi,\mathscr{T}_\Phi)=(\{(0, V, 0)\},\{(U, 0, 0)\})$. Thus, the class $\mathscr{T}_\Phi$ occurs as the torsion class in $(\mathscr{T}_\Phi,\mathscr{F}_\Phi)=(\{(U, 0, 0)\},\{(0, V, 0)\})$ and as the torsionfree class in $(\mathfrak{C}_\Phi,\mathscr{T}_\Phi)=(\{(0, V, 0)\},\{(U, 0, 0)\})$.
\end{exam}

   \section*{Acknowledgment}
    \begin{paragraph}\noi
This paper results from a research visit by the first author to Makerere University in Uganda, supported by the European Mathematical Society, Committee for Developing Countries (EMS, CDC), and the University of Dodoma, Tanzania. 
   \end{paragraph}

\renewcommand{\bibname}{References}
\nocite{*}
\bibliographystyle{plain}
\bibliography{bib} 
\end{document}